\documentclass{amsart}
\usepackage{graphicx,pstricks,pst-plot}
\usepackage{hyperref, color}
\usepackage{amsfonts}
\usepackage[portuges,english]{babel}
\usepackage{graphicx}
\usepackage{amscd,color}
\usepackage{amsmath}
\usepackage{amssymb}

\usepackage{float}
\usepackage{fancyhdr}
\usepackage[
lmargin=3.0cm,
rmargin=3.0cm,
tmargin=3.0cm,
bmargin=3.0cm
]{geometry}

\makeatletter
\@namedef{subjclassname@2010}{%
	\textup{2010} Mathematics Subject Classification}
\makeatother

\theoremstyle{plain}

\newtheorem{definition}{\bf Definition}

\newtheorem{remark}{Remark}

\newtheorem{theorem}{\bf Theorem}
\numberwithin{equation}{section}

\begin{document}
	
	\title[Stereographic Projections on Some Quadric Surfaces]{Stereographic Projections on Some Quadric Surfaces}
	
	\author[W.F.C. Barboza, T.F. Cruz, R.B. Leal]{W.F.C. Barboza$^{\ast}$, T.F. Cruz and R.B. Leal}
	
	\address{$^1$ Unidade Acadêmica de Matem\'atica, Universidade Federal de Campina Grande, 58.429-970 Campina Grande, Para\'iba, Brazil.}
	\email{weiller@mat.ufcg.edu.br}
	\email{thiagofc781@gmail.com}
	\email{rubensbarretoleal@gmail.com}
	
	
	
	\thanks{$^\ast$ Corresponding author.}
	
	\begin{abstract}
	In this work, we present an adaptation of the classical stereographic projection, originally formulated for the sphere, now considering the context of the ellipsoid and the elliptic paraboloid. We begin by constructing the stereographic projections for both quadric surfaces separately, analyzing the geometric particularities of each surface and the challenges arising from their variable curvatures and, in the case of the paraboloid, its non-compactness. In the final part of the work, we establish results concerning the eccentricities, curvatures, arc length, and areas of the ellipses formed by the intersection of the quadrics with horizontal sections and their corresponding projections onto the plane-xy.
	\end{abstract}
	
	\maketitle
	
	\section{Introduction}
	
		The stereographic projection is a classical geometric construction that defines a mapping between the sphere (with one point removed) and the plane. Its use dates back to antiquity, having been employed by Hipparchus and Ptolemy for representing the celestial sphere on flat maps. During the Middle Ages and the Renaissance, instruments such as the astrolabe made use of this projection. However, its formal treatment was only consolidated with the development of analytic geometry and differential geometry, particularly in the 18th and 19th centuries, for more details about this see Gentili, Simonutti and Struppa in \cite{gentili}, Yavetz in \cite{yavetz} and Sidoli, Bergreen in \cite{sidoli}.

	From a mathematical perspective, the stereographic projection of the unit sphere $S^2 \subset \mathbb{R}^3$ given by
	$$
	S^2 = \{ (x,y,z) \in \mathbb{R}^3 \quad ; \quad x^2 + y^2 + z^2 = r^2 \}, 
	$$
	consists in associating to each point $ P $ on the sphere (except the north pole $ N = (0, 0, r) $) a point $ Q = (u,v,0) $ on the plane $ z = 0 $, defined as the intersection between this plane and the line passing through $ N $ and $P$ (See \cite{docarmo}). Let $ P = (x, y, z) \in S^2 \setminus \{N\} $. The stereographic projection maps $ P $ at the point of the xy plano $Q$ is the map
	\begin{equation}
		\begin{array}{rcl}
			\varphi^{-1} & : & S^2 \setminus \{N \} \longrightarrow  \mathbb{R}^2 \\
			&  & (x,y,z) \longmapsto \varphi^{-1}(x,y,z) = \left( \dfrac{rx}{r-z}, \dfrac{ry}{r - z} \right).
		\end{array}
	\end{equation}
	
	The inverse mapping can also be written explicitly, allowing the establishment of a diffeomorphism between $ S^2 \setminus \{N\} $ and $\mathbb{R}^2 $. More specifically, we know that the inverse application of the parameterization of the stereographic projection of the sphere is given by
	\begin{equation}\label{parametriza esfera}
		\begin{array}{rcl}
			\varphi & : & \mathbb{R}^2 \longrightarrow S^2 \setminus \{N \} \\
			&  & (u,v) \longmapsto \varphi(u,v) = \left( \dfrac{2r^2 u}{r^2 + u^2 + v^2}, \dfrac{2r^2 v}{r^2 + u^2 + v^2}, r \dfrac{u^2 + v^2 - r^2}{r^2 + u^2 + v^2} \right).
		\end{array}
	\end{equation}
	
	This projection has several remarkable properties. First, it is a bijective and differentiable map with a differentiable inverse that is, a local diffeomorphism. Moreover, it preserves angles between intersecting curves, making it a conformal mapping. This property makes it useful not only in differential geometry, but also in complex analysis, where it allows for the compactification of the complex plane $\mathbb{C}$, identifying it with the Riemann sphere, for more see \cite{carroll}. In the context of the theory of regular surfaces, the stereographic projection provides a differentiable parametrization of the sphere (excluding one point) onto the plane, whose differential is injective at every point in its domain. This satisfies the criteria of the definition of a regular surface, as established by authors such as Do Carmo~\cite{docarmo}, and enables the study of local properties of the sphere through planar coordinates.
	
	In this work, we initially propose a generalization of the stereographic construction used for the sphere to the case of the ellipsoid a surface defined by a second-order quadratic equation and characterized by variable curvature. Constructing a stereographic projection for the ellipsoid involves new challenges, both geometric and analytic, as spherical symmetry is no longer present. Our goal in the section \ref{secao elipsoide} is to construct a mapping that parametrizes the ellipsoid in a manner analogous to the stereographic projection of the sphere, analyzing its properties and verifying whether the resulting parametrization satisfies the conditions of a regular surface. This approach opens up possibilities for conformal visualization of ellipsoids and for the local study of their differential properties. In a second part of this study in section \ref{secao paraboloide}, We also developed the stereographic projection for the case of the elliptic paraboloid using a process entirely analogous to that of the ellipsoid. In the third and final part, we were able to establish interesting results arising from the expressions of the stereographic projection maps in both cases: for the ellipsoid and the elliptic paraboloid. In particular, we obtain the following results: Theorem \ref{teorema1}, \ref{teorema2}, \ref{teorema3}, \ref{teorema4}, \ref{teorema5}, \ref{teorema6}, \ref{teorema7}, \ref{teorema8}. Where we relate eccentricities, curvatures, arc lengths, and areas between the ellipses formed by horizontal sections and their respective projections through applications \eqref{inversa_elipsoide} and \eqref{inversa_paraboloide}.
	
\section{Initial Settings}\label{section2}

The theory of regular surfaces is one of the cornerstones of differential geometry, focusing on the local and global study of smoothly immersed surfaces in three-dimensional space $\mathbb{R}^3$. We now present some preliminary definitions concerning curves in $\mathbb{R}^2$ from \cite{docarmo}, followed by the formal definition of regular surfaces, along with results that will be addressed in this work.

\vspace{0.3cm}

\begin{definition}
	A parametrized differentiable curve in a differentiable map $\alpha: I \rightarrow \mathbb{R}^2$ of an open interval $I=(a,b)$ of the real line $\mathbb{R}$ into $\mathbb{R}^2$.
\end{definition}

\noindent The word \textit{differentiable} in this definition means that $\alpha$ is a correspondence which maps each $t \in I$ into a point $ \alpha(t) = (x(t), y(t)) \in \mathbb{R}^2 $ in such a way that the functions $x(t)$, $y(t)$ are differentiable. The variable $t$ is called the \textit{parameter} of the curve. The word \textit{interval} is taken in a generalized sense, so that we do not exclude the cases $a = -\infty$, $b = +\infty$.

If we denote by $x'(t)$ the first derivative of $x$ at the point $t$ and use similar notations for the functions $y$, the vector 
\[
\alpha'(t) = (x'(t), y'(t))  \in \mathbb{R}^2
\]
is called the \textit{tangent vector} (or \textit{velocity vector}) of the curve $\alpha$ at $t$. The image set $\alpha(I) \subset \mathbb{R}^2$ is called the \textit{trace} of $\alpha$. 

\begin{definition}
	A parametrized differentiable curve $\alpha: I \longrightarrow \mathbb{R}^2$ is said to be regular if $\alpha'(t)\neq 0$ for all $t \in I$.
\end{definition}

\begin{definition}
	Given $t \in I$, the arc length of a regular parametrized curve $\alpha: I \longrightarrow \mathbb{R}^2$, from the point $t_0 \in I$, is by definition
	\begin{equation}\label{definicao comprimento}
		\mathcal{L}(t) = \int_{t_0}^t |\alpha'(t)|dt.
	\end{equation}
\end{definition}

\begin{definition}
	Let $\alpha: I \longrightarrow \mathbb{R}^2$ be a plane curve $\alpha(t)=(x(t), y(t)),$ the signed curvature of $\alpha$ at $t$ is given by
	\begin{equation}\label{definicao curvatura}
		k(t) = \dfrac{x'(t) y''(t) - x''(t) y'(t)}{[ (x'(t))^2 + (y'(t))^2 ]^{3/2}}.   
	\end{equation}
\end{definition}

\begin{definition}\label{definicao superficie}
	A subset $S \subset \mathbb{R}^3$ is a regular surface if, for each $p \in S$, there exists a neighborhood $V \subset \mathbb{R}^3$ and a map $X: U \longrightarrow V \cap S$ of an open set $U \subset \mathbb{R}^2$ onto $V \cap S \subset \mathbb{R}^3$ such that
	\begin{enumerate}
		\item[a)] $x$ is differentiable. This means that if we write
		$$
		X(u,v) = (x(u,v), y(u,v), z(u,v)), \quad (u,v) \in U,
		$$
		the functions $x(u,v),y(u,v),z(u,v)$ have continuous partial derivatives of all orders in $U$.
		
		\item[b)] $X$ is a homeomorphism. Since $X$ is continuous by condition $1$, this is, $X$ has an inverse $X^{-1}: V\cap S \longrightarrow U$ which is continuous.
		
		\item[c)] For each $q \in U$, the differential map $dX_q: \mathbb{R}^2 \longrightarrow \mathbb{R}^3$ is one-to-one.
		
	\end{enumerate}
	
\end{definition}

The mapping $X$ is called a parametrization or a system of (local) coordinates in (a neighborhood of) p. The neighborhood $V \cap S$ of $p$ is called a coordinate neighborhood. we will see in more detail what item 3 of the definition means. Let's calculate now the matrix of the linear map \( dX_q \) in the canonical bases $\{ e_1=(1,0), e_2=(0,1) \}$ of $\mathbb{R}^2$, with coordinates $(u, v)$, and $\{ f_1 = (1, 0, 0), f_2 = (0, 1, 0), f_3 = (0, 0, 1) \}$ of $\mathbb{R}^3$, with coordinates $(x, y, z)$.

Let $q = (u_0, v_0)$. The vector $e_1$ is tangent to the curve $u \mapsto (u, v_0)$, whose image by $X$ is the curve on the surface $ S $,
$$
u \mapsto (x(u, v_0), y(u, v_0), z(u, v_0)),
$$
called the coordinate curve $ v = v_0 $. The tangent vector of this curve at $X(q)$ is the vector
$$
dX_q(e_1) = \frac{\partial X}{\partial u}(q) =
\left( \frac{\partial x}{\partial u}(q), \frac{\partial y}{\partial u}(q), \frac{\partial z}{\partial u}(q) \right).
$$
Similarly, the tangent vector to the coordinate curve $u = u_0$,
$$
v \mapsto (x(u_0, v), y(u_0, v), z(u_0, v)),
$$
the image by $X$ of the curve $v \mapsto (u_0, v)$, is the vector
$$
dX_q(e_2) = \frac{\partial X}{\partial v}(q) =
\left( \frac{\partial x}{\partial v}(q), \frac{\partial y}{\partial v}(q), \frac{\partial z}{\partial v}(q) \right).
$$
Therefore, the matrix of the linear map $dX_q$ (which we denote with the same notation for simplicity) in the canonical bases of $\mathbb{R}^2$ and $ \mathbb{R}^3 $ is
$$
dX_q =
\begin{pmatrix}
	\frac{\partial x}{\partial u}(q) & \frac{\partial x}{\partial v}(q) \\ ~ \\
	\frac{\partial y}{\partial u}(q) & \frac{\partial y}{\partial v}(q) \\ ~ \\
	\frac{\partial z}{\partial u}(q) & \frac{\partial z}{\partial v}(q)
\end{pmatrix}.
$$
The Condition (3), from Definition 1, tells us that $dX_q: \mathbb{R}^2 \to \mathbb{R}^3$ is one-to-one, which means that the two column vectors of the Jacobian matrix above are linearly independent. In other words, their cross product is non-zero:
\begin{equation}\label{produto vetorial nao nulo}
	\frac{\partial X}{\partial u}(q) \wedge \frac{\partial X}{\partial v}(q) \neq (0,0,0),
\end{equation}
or equivalently, that one of the 2x2 minors of the matrix of $dX_q$, that is, one of the determinants:
$$
\frac{\partial(x, y)}{\partial(u, v)}(q) =
\begin{vmatrix}
	\frac{\partial x}{\partial u}(q) & \frac{\partial x}{\partial v}(q) \\ \\
	\frac{\partial y}{\partial u}(q) & \frac{\partial y}{\partial v}(q)
\end{vmatrix}, 
\quad \quad
\frac{\partial(y, z)}{\partial(u, v)}(q) =
\begin{vmatrix}
	\frac{\partial y}{\partial u}(q) & \frac{\partial y}{\partial v}(q) \\ \\
	\frac{\partial z}{\partial u}(q) & \frac{\partial z}{\partial v}(q)
\end{vmatrix},
$$
or
$$
\frac{\partial(x, z)}{\partial(u, v)}(q) =
\begin{vmatrix}
	\frac{\partial x}{\partial u}(q) & \frac{\partial x}{\partial v}(q) \\ \\
	\frac{\partial z}{\partial u}(q) & \frac{\partial z}{\partial v}(q)
\end{vmatrix},
$$
is different from zero.


This section presents the construction of stereographic projections for the ellipsoid and the elliptic paraboloid. The geometric and analytical procedures involved in obtaining these projections will be explored, highlighting their specific features and the adaptation of the classical stereographic projection method to these surfaces.

\section{Stereographic Projection of the Ellipsoid}\label{secao elipsoide}

Let $\Sigma$ be an ellipsoid centered at the origin of $\mathbb{R}^3$ given by
\begin{equation}\label{elipsoide}
	\Sigma = \left\{ (x,y,z) \in \mathbb{R}^3 \quad ; \quad \dfrac{x^2}{a^2} + \dfrac{y^2}{b^2} + \dfrac{z^2}{c^2} = 1 \right\},
\end{equation}
where \(a, b, c\) are positive real numbers. To perform the stereographic projection of the ellipsoid onto the plane \(z=0\), we begin by considering the equation of the line passing through the North pole of the ellipsoid, \(N(0,0,c)\), and its projection \(Q(u,v,0)\) on the plane \(z=0\).

\begin{figure}[H]
	\centering
	\caption{Ellipsoid stereographic projection equation.}
	\includegraphics[width=10cm]{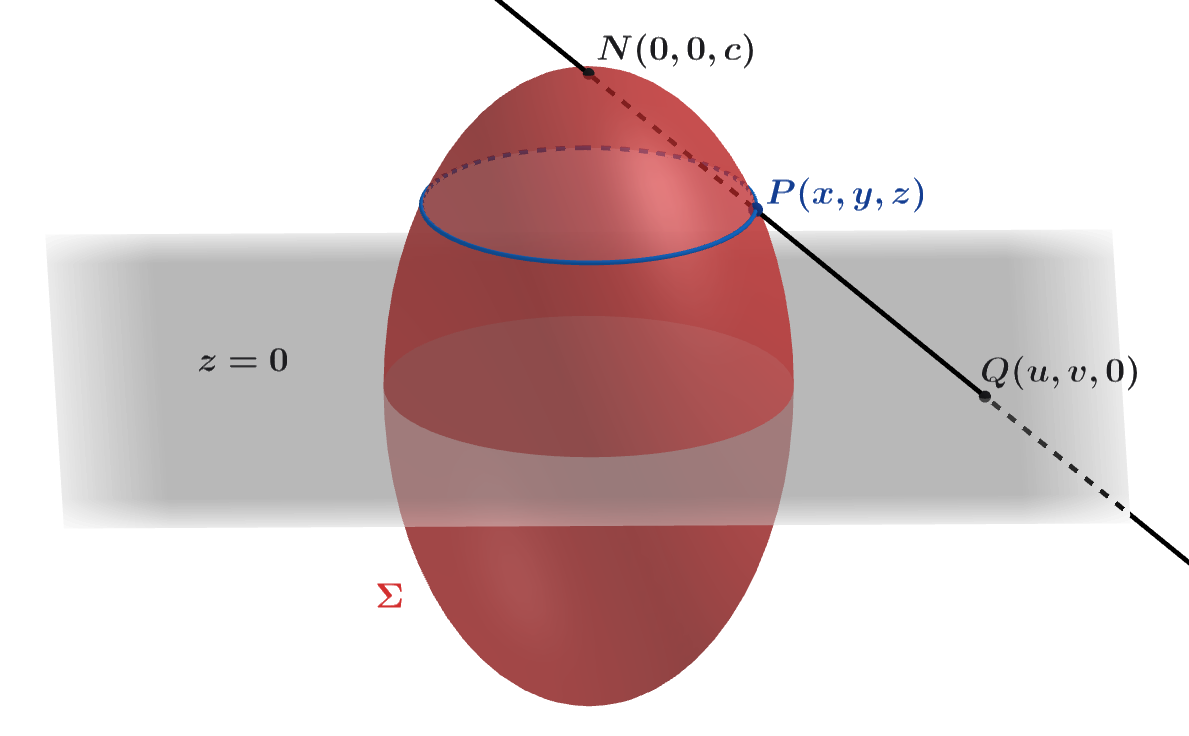}
	\label{FOTO1}\\
	\text{Source: Prepared by the authors.}
\end{figure}
Note that the vector equation of the line passing through the point $N(0,0,c)$ and the projected point on the plane $Q(u,v,0)$ is given by
$$
r: (x,y,z) = (0,0,c) + t(u,v,-c), \quad t \in \mathbb{R},
$$
in the parametric equation, we obtain the equation of the line as
\begin{equation}\label{reta_elipsoide}
	r: 
	\begin{cases}
		x = tu \\
		y = tv \\
		z = (1 - t)c
	\end{cases}, \quad t \in \mathbb{R}.
\end{equation}
Thus, for each point \( (u,v) \in \mathbb{R}^2 \), we can find the intersection point of the sphere with the stereographic projection. Therefore, we may consider the following function:
\begin{equation}\label{parametrizacao_elipsoide}
	\begin{array}{rcl}
		\varphi & : & \mathbb{R}^2 \longrightarrow \Sigma \setminus \{N\} \\
		&  & (u,v) \longmapsto \varphi(u,v) = N + t(Q - N) = (tu, tv, (1 - t)c).
	\end{array}
\end{equation}
In order to find the parameter \( t \) corresponding to the point \( P(x,y,z) \) on the ellipsoid that has been projected onto the plane, we need to determine the intersection between the ellipsoid \eqref{elipsoide} and the line \eqref{reta_elipsoide}. This is done by substituting the coordinates \( x, y, z \) from the line \eqref{reta_elipsoide} into the equation of the ellipsoid \eqref{elipsoide}, we have,
$$
\begin{aligned}
	\dfrac{(ut)^2}{a^2} + \dfrac{(vt)^2}{b^2} + \dfrac{((1-t)c)^2}{c^2} = 1 
	& \Longrightarrow  \dfrac{u^2 t^2}{a^2} + \dfrac{v^2 t^2}{b^2} + \dfrac{(1-t)^2c^2}{c^2} = 1  \\
	& \Longrightarrow \dfrac{u^2 t^2}{a^2} + \dfrac{v^2 t^2}{b^2} + 1 - 2t + t^2 = 1 \\
	& \Longrightarrow b^2 u^2 t^2 + a^2 v^2 t^2 + a^2 b^2 - 2ta^2b^2 + a^2 b^2t^2 = a^2 b^2 \\
	& \Longrightarrow b^2 u^2 t^2 + a^2 v^2 t^2 + a^2 b^2t^2 = 2ta^2b^2 \\
	& \Longrightarrow t^2 ( b^2 u^2 + a^2 v^2 + a^2 b^2 ) = 2ta^2b^2,
\end{aligned}
$$
since \( t \neq 0 \) (because for \( t = 0 \), we obtain the point \( N(0,0,c) \)), we can divide both sides of the equation by \( t \) and isolate the parameter \( t \) corresponding to the point on the ellipsoid, obtaining
\[
t = \dfrac{2a^2b^2}{b^2 u^2 + a^2 v^2 + a^2 b^2}.
\]
Substituting the parameter \( t \) into the equation of the line, we obtain from \eqref{parametrizacao_elipsoide} that
\begin{equation}\label{par_elips}
	\varphi(u,v) = \left(  \dfrac{2a^2 b^2 u}{b^2u^2 + a^2v^2 + a^2 b^2} , \dfrac{2a^2 b^2 v}{b^2u^2 + a^2v^2 + a^2 b^2} , \dfrac{b^2 u^2 + a^2 v^2 - a^2 b^2}{b^2u^2 + a^2v^2 + a^2 b^2} c \right),
\end{equation}
which represents the points on the ellipsoid associated with a fixed point \( (u,v) \in \mathbb{R}^2 \). Note that when $a=b=c=r$ we recover the known parameterization on the sphere \eqref{parametriza esfera}. Now we show that \( \varphi \) is a parametrization of \( \Sigma \setminus \{N\} \). That is, \( \varphi \) must satisfy the following conditions from the definition \ref{definicao superficie}.

\noindent\textbf{(a)} \( \varphi \) is differentiable; In fact, in first to place observe that \( \varphi \) is differentiable since its coordinate functions are differentiable and 
$$
b^2u^2 + a^2v^2 + a^2 b^2 > 0.
$$ 

\noindent\textbf{(c)} The differential map $d\varphi(u,v)$ is one-to-one. In fact, To verify this, from \eqref{produto vetorial nao nulo} simply show that
$$
\frac{\partial\varphi(u,v)}{\partial u} \times \frac{\partial\varphi(u,v)}{\partial v} \neq (0,0,0),
$$
that is, the above vectors are linearly independent. From equation \eqref{par_elips}, and defining
$$
D = D(u,v) = b^2u^2 + a^2v^2 + a^2b^2,
$$
the parametrization becomes
$$
\varphi(u,v) = \left( \frac{2a^2b^2u}{D}, \; \frac{2a^2b^2v}{D}, \; \frac{(b^2u^2 + a^2v^2 - a^2b^2)c}{D} \right).
$$
The partial derivative with respect to $u$ is given by
$$
\frac{\partial \varphi}{\partial u}  = \left( 
\frac{\partial}{\partial u}\left(\frac{2a^2b^2u}{D}\right), 
\frac{\partial}{\partial u}\left(\frac{2a^2b^2v}{D}\right), 
\frac{\partial}{\partial u}\left(\frac{(b^2u^2 + a^2v^2 - a^2b^2)c}{D}\right), 
\right) 
$$
by calculating each component we obtain,
\begin{equation}\label{vetor1}
	\frac{\partial \varphi}{\partial u} = \left( 
	\frac{2a^2b^2(D - 2b^2u^2)}{D^2}, 
	\frac{-4a^2b^4uv}{D^2}, 
	\frac{4a^2b^4uc}{D^2} 
	\right),
\end{equation}
working in a similar way for the derivative with respect to v, we get
$$
\frac{\partial \varphi}{\partial v} = \left( 
\frac{\partial}{\partial v}\left(\frac{2a^2b^2u}{D}\right), 
\frac{\partial}{\partial v}\left(\frac{2a^2b^2v}{D}\right), 
\frac{\partial}{\partial v}\left(\frac{(b^2u^2 + a^2v^2 - a^2b^2)c}{D}\right) 
\right),
$$
in the same way, we get
\begin{equation}\label{vetor2}
	\frac{\partial \varphi}{\partial v} = \left( 
	\frac{-4a^4b^2uv}{D^2}, 
	\frac{2a^2b^2(D - 2a^2v^2)}{D^2}, 
	\frac{4a^4b^2vc}{D^2} 
	\right).
\end{equation}
Now calculating the vector product between \eqref{vetor1} and \eqref{vetor2} we have
$$
\frac{\partial \varphi}{\partial u} \times \frac{\partial \varphi}{\partial v} = \large \begin{vmatrix} 
	\vec{i} & \vec{j} & \vec{k} \\
	\frac{2a^2b^2(D - 2b^2u^2)}{D^2} & \frac{-4a^2b^4uv}{D^2} & \frac{4a^2b^4uc}{D^2} \\
	\frac{-4a^4b^2uv}{D^2} & \frac{2a^2b^2(D - 2a^2v^2)}{D^2} & \frac{4a^4b^2vc}{D^2}
\end{vmatrix}.
$$
Thus, the \( x \)-component is given by
$$
\left( \frac{-4a^2b^4uv}{D^2} \cdot \frac{4a^4b^2vc}{D^2} - \frac{4a^2b^4uc}{D^2} \cdot \frac{2a^2b^2(D - 2a^2v^2)}{D^2} \right)\vec{i} = \frac{-8a^4b^6uc}{D^3}\vec{i},
$$
the \( y \)-component is
$$
\left( \frac{4a^2b^4uc}{D^2} \cdot \frac{-4a^4b^2uv}{D^2} - \frac{2a^2b^2(D - 2b^2u^2)}{D^2} \cdot \frac{4a^4b^2vc}{D^2} \right)\vec{j} = \frac{-8a^6b^4vc}{D^3}\vec{j},
$$
finally, the \( z \)-component is
$$
\left( \frac{2a^2b^2(D - 2b^2u^2)}{D^2} \cdot \frac{2a^2b^2(D - 2a^2v^2)}{D^2} - \frac{-4a^2b^4uv}{D^2} \cdot \frac{-4a^4b^2uv}{D^2} \right)\vec{k} = \frac{4a^4b^4(a^2b^2 - b^2u^2 - a^2v^2)}{D^3}\vec{k}.
$$
Therefore, the final cross product is given by
$$
\frac{\partial \varphi}{\partial u} \times \frac{\partial \varphi}{\partial v} = \frac{4a^4b^4}{D^3} \left( -2b^2uc,\; -2a^2vc,\; a^2b^2 - b^2u^2 - a^2v^2 \right),
$$
observe that this cross product is different from \( (0,0,0) \) for all \( (u,v) \in \mathbb{R}^2 \), since the first and second components are zero if and only if \( u = 0 \) and \( v = 0 \), respectively. However, if \( u = v = 0 \), then the third component is nonzero, since we have \( a^2b^2 > 0 \). Thus, the vectors above are linearly independent, and therefore the differential map \( d\varphi(u,v) \) is one-to-one.

\noindent\textbf{(b)} \( \varphi \) is a homeomorphism: Let us see if we can calculate an inverse for \( \varphi \). Given
\( p = (x,y,z) \in \Sigma \setminus \{N\} \), to determine $ \varphi^{-1}(p) \in \mathbb{R}^2$, we just need to find the intersection point of the line
$$
r: N + t(p - N) = (tx, ty, c + t(z - c)),
$$
with the plane \( z = 0 \), which occurs when
$$
c + t(z - c) = 0 \quad \Longrightarrow \quad t = \frac{c}{c - z}.
$$
Thus, a candidate for the inverse is
\begin{equation}\label{inversa_elipsoide}
	\begin{array}{rcl}
		\varphi^{-1} & : & \Sigma \setminus \{N \} \longrightarrow  \mathbb{R}^2 \\
		&  & \quad (x,y,z) \longmapsto \varphi^{-1}(x,y,z) = \left( \dfrac{cx}{c-z}, \dfrac{cy}{c - z} \right).
	\end{array}
\end{equation}
Let us verify if \( \varphi^{-1} \circ \varphi = Id = \varphi \circ \varphi^{-1} \). Let \( (u,v) \in \mathbb{R}^2 \). Taking 
\[ D(u,v) = b^2u^2 + a^2v^2 + a^2b^2, \] 
and knowing that
\begin{equation}\label{c-z}
	c - z = c - \frac{b^2u^2 + a^2v^2 - a^2b^2}{D}c
	= c \left( 1 - \frac{b^2u^2 + a^2v^2 - a^2b^2}{D} \right)
	= c \cdot \frac{2a^2b^2}{D},
\end{equation}
we have
\begin{align*}
	(\varphi^{-1} \circ \varphi)(u,v) &= \varphi^{-1}(\varphi(u,v)) \\
	&= \varphi^{-1}\left(  \dfrac{2a^2 b^2 u}{D} , \dfrac{2a^2 b^2 v}{D} , \dfrac{b^2 u^2 + a^2 v^2 - a^2 b^2}{D} c \right) \\
	&= \left( \frac{c\frac{2a^2b^2u}{D}}{c\frac{2a^2b^2}{D}}, \frac{c\frac{2a^2b^2v}{D}}{c\frac{2a^2b^2}{D}} \right) = (u,v).
\end{align*}
On the other hand, let's verify the other equality:
\begin{align*}
	(\varphi \circ \varphi^{-1})(x,y,z) &= \varphi(\varphi^{-1}(x,y,z)) = (x,y,z).
\end{align*}
We will calculate each coordinate separately. Note that given \( (x,y,z) \in \Sigma \setminus \{N\} \), we have
\begin{align*}
	(\varphi \circ \varphi^{-1})(x,y,z) &= \varphi(\varphi^{-1}(x,y,z)) = \varphi\left( \frac{cx}{c - z}, \frac{cy}{c - z} \right).
\end{align*}
Recall that
\begin{align*}
	\varphi(u,v) &= \left( \frac{2a^2b^2u}{D}, \frac{2a^2b^2v}{D}, \frac{b^2u^2 + a^2v^2 - a^2b^2}{D} \cdot c \right),
\end{align*}
where \( D = b^2u^2 + a^2v^2 + a^2b^2 \). Let's calculate the value of the first coordinate, \( \vec{i} \). We substitute \( u = \dfrac{cx}{c - z} \) into the first component of \( \varphi(u,v) \)
\begin{align*}
	\frac{2a^2b^2u}{D} &= \frac{2a^2b^2 \cdot \dfrac{cx}{c - z}}{D} = \frac{2a^2b^2cx}{(c - z)D},
\end{align*}
using the relation \( c - z = \dfrac{2a^2b^2}{D}c \) found in \eqref{c-z}, we obtain the \( i \)-coordinate
\begin{align*}
	\frac{2a^2b^2u}{D} &= \frac{2a^2b^2cx}{\dfrac{2a^2b^2}{D}c \cdot D} = \frac{2a^2b^2cx}{2a^2b^2c} = x,
\end{align*}
the coordinate \( \vec{j} \) follows the same steps as for the coordinate \( \vec{i} \), thus obtaining that the coordinate \( \vec{j} \) is \( y \). Let \( p \) denote the third coordinate of \( \varphi \) and let's show that \( p = z \). We know that \( p \) must satisfy
$$
p = \frac{b^2u^2 + a^2v^2 - a^2b^2}{D} \cdot c,
$$
substituting \( u = \dfrac{cx}{c - z} \) and \( v = \dfrac{cy}{c - z} \), we get
$$
p = \frac{b^2\left( \dfrac{cx}{c-z} \right)^2 + a^2\left( \dfrac{cy}{c-z} \right)^2 - a^2b^2 }{D} \cdot c
= \frac{ \dfrac{b^2c^2x^2 + a^2c^2y^2}{(c-z)^2} - a^2b^2 }{D} \cdot c,
$$
from the equation \eqref{c-z}, we know that \( c - z = \dfrac{2a^2b^2}{D}c \), so
$$
(c - z)^2 = \left( \dfrac{2a^2b^2}{D}c \right)^2 = \dfrac{4a^4b^4c^2}{D^2},
$$
substituting this in the denominator:
$$
p = \frac{ \dfrac{b^2c^2x^2 + a^2c^2y^2}{\dfrac{4a^4b^4c^2}{D^2}} - a^2b^2 }{D} \cdot c
=\left( \dfrac{(b^2c^2x^2 + a^2c^2y^2)D^2}{4a^4b^4c^2} - a^2b^2 \right) \cdot \dfrac{c}{D}.
$$
Now, we put everything over a common denominator:
$$
p = \left( \dfrac{(b^2c^2x^2 + a^2c^2y^2)D^2 - 4a^6b^6c^2}{4a^4b^4c^2} \right) \cdot \dfrac{c}{D}
= \dfrac{(b^2c^2x^2 + a^2c^2y^2)D^2 - 4a^6b^6c^2}{4a^4b^4c D}.
$$
Since \( (x,y,z) \in E \setminus \{N\} \), the equation \( \dfrac{x^2}{a^2} + \dfrac{y^2}{b^2} + \dfrac{z^2}{c^2} = 1 \) is satisfied, so
$$
b^2c^2x^2 + a^2c^2y^2 = a^2b^2(c^2 - z^2),
$$
substituting this in the numerator, we get
$$
p = \dfrac{a^2b^2(c^2 - z^2)D^2 - 4a^6b^6c^2}{4a^4b^4c D}
= \dfrac{a^2b^2 \left[ (c^2 - z^2)D^2 - 4a^4b^4c^2 \right]}{4a^4b^4c D}= \dfrac{(c^2 - z^2)D^2 - 4a^4b^4c^2}{4a^2b^2c D},
$$
from the equation \eqref{c-z}, for \( c - z = \dfrac{2a^2b^2}{D}c \), isolating \( D \), we get
$$
D = \dfrac{2a^2b^2c}{c - z},
$$
substituting \( D \) in the last expression for \( p \):
$$
p = \dfrac{(c^2 - z^2)\left( \dfrac{2a^2b^2c}{c - z} \right)^2 - 4a^4b^4c^2}{4a^2b^2c \cdot \dfrac{2a^2b^2c}{c - z}}
= \dfrac{(c^2 - z^2) \cdot \dfrac{4a^4b^4c^2}{(c - z)^2} - 4a^4b^4c^2}{\dfrac{8a^4b^4c^2}{c - z}}.
$$
Now, we put everything over a common denominator \( (c - z)^2 \):
$$
p = \dfrac{ \dfrac{4a^4b^4c^2(c^2 - z^2) - 4a^4b^4c^2(c - z)^2}{(c - z)^2} }{ \dfrac{8a^4b^4c^2}{c - z}}
= \dfrac{4a^4b^4c^2 \left[ (c^2 - z^2) - (c - z)^2 \right]}{8a^4b^4c^2 (c - z)},
$$
now, canceling \( 4a^4b^4c^2 \), we get
$$
p = \dfrac{(c^2 - z^2) - (c^2 - 2cz + z^2)}{2(c - z)},
$$
consequently,
$$
p = \dfrac{c^2 - z^2 - c^2 + 2cz - z^2}{2(c - z)} = \dfrac{-2z^2 + 2cz}{2(c - z)} = \dfrac{-2z(z - c)}{2(c - z)} = \dfrac{2z(c - z)}{2(c - z)} = z.
$$

In this way, the stereographic projection map of the ellipsoid \eqref{parametrizacao_elipsoide} associates all points of the ellipsoid to a point in the \( xy \)-plane. As done for the sphere, if one wishes to cover the entire ellipsoid with parametrizations of this type, it is sufficient to perform a completely analogous construction, now considering the map defined on $ \Sigma \setminus S$, where \( S = (0, 0, -c) \). From this, it becomes clear that the entire ellipsoid can be covered using these two parametrizations and from the definition \ref{definicao superficie}, we can conclude that $\Sigma$ is a surface. However, our focus here is to study the interesting properties that this projection map provides. Let us now examine how the construction works in the case of the elliptic paraboloid. 

We state in advance that the entire construction is quite analogous to what we did for the ellipsoid therefore, we shall proceed more briefly.

\section{Stereographic Projection of the Elliptical Paraboloid}\label{secao paraboloide}

Let \( \mathcal{P} \subset \mathbb{R}^3 \) be the elliptical paraboloid with vertex $N$ given by
\begin{equation} \label{paraboloide}
	\mathcal{P} = \left\{ (x,y,z) \in \mathbb{R}^3 \quad ; \quad z = c - \dfrac{x^2}{a^2} - \dfrac{y^2}{b^2}  \right\},  
\end{equation}
where \( a,b,c \in \mathbb{R} \) are positive constants.

To perform the stereographic projection of the elliptic paraboloid onto the \( xy \)-plane (\( z = 0 \)), we start by considering the equation of the line passing through the North pole of the paraboloid, that is, the point \( N(0,0,c) \), and its projection \( Q(u,v,0) \) on the plane \( z=0 \). The vector equation of the line passing through \( N(0,0,c) \) and the point \( Q(u,v,0) \) is given by
\begin{equation}\label{reta_paraboloide}
	r: 
	\begin{cases}
		x = tu \\
		y = tv \\
		z = (1 - t)c
	\end{cases}, \quad t \in \mathbb{R}.
\end{equation}

\begin{figure}[H]
	\centering
	\caption{Elliptic Paraboloid stereographic projection equation.}
	\includegraphics[width=11cm]{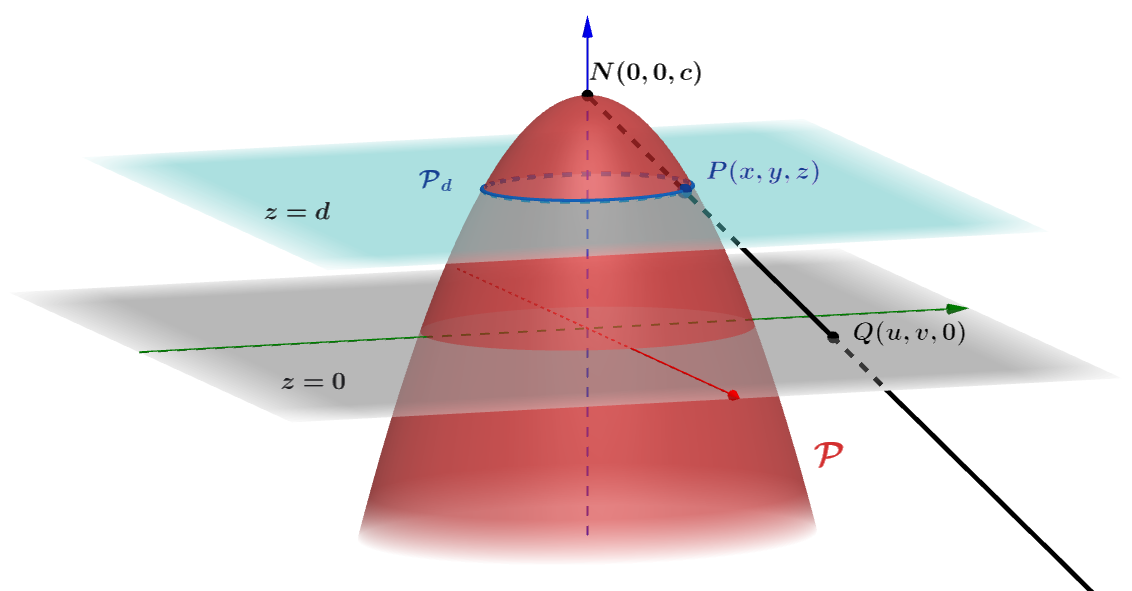}
	\label{FOTO2}\\
	\text{Source: Prepared by the authors.}
\end{figure}

Thus, for each point \( (u,v) \in \mathbb{R}^2 \), we can find the intersection point of the paraboloid with the stereographic projection. We then define the following map:
\begin{equation}\label{parametrizacao_paraboloide}
	\begin{array}{rcl}
		\phi & : & \mathbb{R}^2-\{(0,0)\} \longrightarrow \mathcal{P} \setminus \{N \} \\
		&  & \quad \quad (u,v) \quad \longmapsto ~~ \phi(u,v) = N + t(Q-N) = (tu, tv, (1-t)c).
	\end{array}
\end{equation}
To find the parameter \( t \) for the point \( P(x,y,z) \) on the paraboloid projected from the plane, we substitute the coordinates \( x, y, z \) from the line into the paraboloid equation \eqref{paraboloide}, that is,
$$
\begin{aligned}
	(1-t)c = c - \dfrac{(ut)^2}{a^2} - \dfrac{(vt)^2}{b^2}  &\Longrightarrow  -ct = - \dfrac{u^2 t^2}{a^2} - \dfrac{v^2 t^2}{b^2} 
	\\ &\Longrightarrow  ct = t^2 \left(  \dfrac{u^2}{a^2} + \dfrac{v^2}{b^2} \right),
\end{aligned}
$$
since \( t \neq 0 \) (as \( t = 0 \) yields the North pole \( N(0,0,c) \)), we can divide both sides by \( t \) and isolate, we get
$$
t = \dfrac{c}{\dfrac{u^2}{a^2} + \dfrac{v^2}{b^2}} = \dfrac{ca^2b^2}{b^2 u^2 + a^2v^2},
$$
substituting this value of \( t \) into \eqref{parametrizacao_paraboloide}, we obtain
\begin{equation}\label{phi}
	\phi(u,v) = \left( \dfrac{a^2b^2cu}{b^2u^2+a^2v^2}, \dfrac{a^2b^2cv}{b^2u^2+a^2v^2}, \dfrac{(b^2u^2+a^2v^2 - a^2b^2c)c}{b^2u^2+a^2v^2}\right).
\end{equation}
We now verify that \( \phi \) is a parametrization of \( \mathcal{P} \setminus \{N\} \) that satisfies the conditions of a regular surface. 

\noindent\textbf{(a)} In first to place, note that \( \phi \) is differentiable since \( (u,v) \neq (0,0) \) and its component functions are differentiable.

\noindent\textbf{(c)} We will now show that the differential $d\phi(u,v)$ is one-to-one. To do this, we begin by defining
$$
M(u,v) = b^2u^2 + a^2v^2.
$$

\noindent Rewriting $\phi$, we have
\[
\phi(u,v) = \left( \frac{a^2b^2cu}{M}, \frac{a^2b^2cv}{M}, \frac{(M - a^2b^2c)c}{M} \right),
\]
differentiating $\phi$ with respect to $u$, we obtain
\begin{equation}\label{derivada_u_en}
	\frac{\partial \phi}{\partial u} = \left( \frac{a^2b^2c(M - 2b^2u^2)}{M^2}, \frac{-2a^2b^4cuv}{M^2}, \frac{2a^2b^4c^2u}{M^2} \right),
\end{equation}

\noindent and differentiating with respect to $v$, we get
\begin{equation}\label{derivada_v_en}
	\frac{\partial \phi}{\partial v} = \left( \frac{-2a^4b^2cuv}{M^2}, \frac{a^2b^2c(M - 2a^2v^2)}{M^2}, \frac{2a^4b^2c^2v}{M^2} \right).
\end{equation}

\noindent Let us analyze the third coordinate of the cross product between \eqref{derivada_u_en} and \eqref{derivada_v_en}:
$$
\frac{a^2b^2c(M - 2b^2u^2)}{M^2} \cdot \frac{a^2b^2c(M - 2a^2v^2)}{M^2} 
- \frac{-2a^2b^4cuv}{M^2} \cdot \frac{-2a^4b^2cuv}{M^2} 
= -\frac{a^4b^4c^2}{M^2},
$$
note that this third coordinate is strictly negative. Therefore, the cross product between \eqref{derivada_u_en} and \eqref{derivada_v_en} is non-zero:
$$
\frac{\partial \phi}{\partial u} \times \frac{\partial \phi}{\partial v} \neq (0,0,0),
$$
hence, the differential map $d\phi(u,v)$ is one-to-one.

\noindent\textbf{(b)} We now show that $\phi$ is a homeomorphism. Since $\phi$ is differentiable, it is continuous. Analogously to the ellipsoid case, by finding the point of intersection between the line  
$$
r: N + t(p - N) = (tx, ty, c + t(z - c)),
$$
and the plane $z = 0$, where $p = (x, y, z) \in \mathcal{P} \setminus \{N\}$, we obtain the parameter $t$ that leads to the following candidate for the inverse function of $\phi$ is

\begin{equation}\label{inversa_paraboloide}
	\begin{array}{rcl}
		\phi^{-1} & : & \mathcal{P} \setminus \{N \} \longrightarrow  \mathbb{R}^2 \setminus \{(0,0 \} \\
		&  & \quad (x,y,z) \longmapsto \phi^{-1}(x,y,z) = \left( \dfrac{cx}{c-z}, \dfrac{cy}{c - z} \right).
	\end{array}
\end{equation}

\noindent Let $(u,v) \in \mathbb{R}^2 \setminus \{(0,0)\}$ thus, we obtain
\begin{align*}
	(\phi^{-1} \circ \phi)(u,v) &= \phi^{-1}(\phi(u,v)) \\
	&= \phi^{-1} \left( \frac{a^2b^2cu}{M}, \frac{a^2b^2cv}{M}, \frac{(M - a^2b^2c)c}{M} \right) \\
	&= \left( \frac{\dfrac{a^2b^2c^2u}{M}}{c - \dfrac{(M - a^2b^2c)c}{M}}, \frac{\dfrac{a^2b^2c^2v}{M}}{c - \dfrac{(M - a^2b^2c)c}{M}} \right) \\
	&= \left( \frac{\dfrac{a^2b^2c^2u}{M}}{\dfrac{a^2b^2c^2}{M}}, \frac{\dfrac{a^2b^2c^2v}{M}}{\dfrac{a^2b^2c^2}{M}} \right) = (u, v),
\end{align*}
\noindent on the orther hand, let $(x,y,z) \in \mathcal{P} \setminus \{N\}$ then, we have
\begin{align*}
	z = c - \frac{x^2}{a^2} - \frac{y^2}{b^2} 
	&\Longleftrightarrow c - z = \frac{x^2}{a^2} + \frac{y^2}{b^2} \Longleftrightarrow a^2b^2(c - z) = b^2x^2 + a^2y^2,
\end{align*}
where $0 \leq z < c$. Therefore, we get
\begin{align}\label{conta_phi}
	(\phi \circ \phi^{-1} )(x,y,z) &= \phi(\phi^{-1}(x,y,z)) \nonumber \\
	&= \phi \left( \frac{cx}{c - z}, \frac{cy}{c - z} \right) \\
	&= \left( \frac{a^2b^2c \cdot \dfrac{cx}{c - z}}{M(x,y,z)}, \frac{a^2b^2c \cdot \dfrac{cy}{c - z}}{M(x,y,z)}, \frac{(M(x,y,z) - a^2b^2c)c}{M(x,y,z)} \right), \nonumber
\end{align}
where
\begin{align*}
	M(x,y,z) &= a^2\left( \frac{cy}{c - z} \right)^2 + b^2\left( \frac{cx}{c - z} \right)^2 \\
	&= \frac{a^2c^2y^2 + b^2c^2x^2}{(c - z)^2} \\
	&= \frac{(b^2x^2 + a^2y^2)c^2}{(c - z)^2},
\end{align*}
since $(x, y, z) \in \mathcal{P} \setminus \{N\}$, with $0 \leq z < c$, we obtain
\[
M(x,y,z) = \frac{a^2b^2(c - z)c^2}{(c - z)^2} = \frac{a^2b^2c^2}{c - z}.
\]
\noindent Thus, returning to the expression \eqref{conta_phi}, we conclude
\begin{align*}
	\phi(\phi^{-1}(x,y,z)) 
	&= \left( \frac{ \dfrac{a^2b^2c^2x}{c - z} }{ \dfrac{a^2b^2c^2}{c - z} }, \frac{ \dfrac{a^2b^2c^2y}{c - z} }{ \dfrac{a^2b^2c^2}{c - z} }, \frac{ \left( \dfrac{a^2b^2c^2}{c - z} - a^2b^2c \right) c }{ \dfrac{a^2b^2c^2}{c - z} } \right) \\
	&= \left( x, y, \frac{ \left( \dfrac{a^2b^2c^2 - a^2b^2c^2 + a^2b^2cz}{c - z} \right) c }{ \dfrac{a^2b^2c^2}{c - z} } \right) \\
	&= \left( x, y, \frac{a^2b^2c^2z}{a^2b^2c^2} \right) = (x, y, z).
\end{align*}
\noindent Therefore, $\phi^{-1} \circ \phi = \phi \circ \phi^{-1} = \mathrm{Id}$, that is, $\phi^{-1}$ is indeed the inverse of $\phi$. Since $z < c$, we see that the component functions of $\phi^{-1}$ are continuous, and thus $\phi^{-1}$ is continuous. Consequently, $\phi$ is a homeomorphism. Analogously to the ellipsoid case, our goal from this point on is to analyze the interesting properties that the projection map provides in the case of the elliptic paraboloid.

\section{Results and Applications}\label{section4}

At this moment, we discuss some applications involving projections of the ellipsoid and the paraboloid. We show that such projections preserve the eccentricities of the conic sections, which allows us to analyze relevant geometric properties. Furthermore, we establish a relationship between the curvatures of the horizontal sections of these surfaces and the curvatures of their respective projections onto the plane \( z = 0 \).

\vspace{0.5cm}

\begin{center}
	\centering
	\footnotesize
	\text{5.1. RELATIONSHIP BETWEEN THE ECCENTRICITIES}
\end{center}

The eccentricity $e$ of an ellipse is given by $e = {c}/{a},$ where $c$ is the distance from the center to the focus and $a$ is the length of the semi-major axis. The value of $c$ can be computed using $c^2 = a^2 - b^2$, where $b$ is the length of the semi-minor axis. Next, we analyze the behavior of ellipses obtained in the plane $\mathbb{R}^2$ through stereographic projection. Considering the ellipsoid given in \eqref{elipsoide} and the horizontal plane $z = d$, with $d < c$, the intersection is described by:
\begin{equation*}
	\mathcal{E}_d: 
	\begin{cases}
		\dfrac{x^2}{a^2} + \dfrac{y^2}{b^2} + \dfrac{z^2}{c^2} = 1 \\
		z = d
	\end{cases}, \quad d < c.
\end{equation*}

\begin{figure}[H]
	\centering
	\caption{Intersection of the ellipsoid with a plane.}
	\includegraphics[width=11cm]{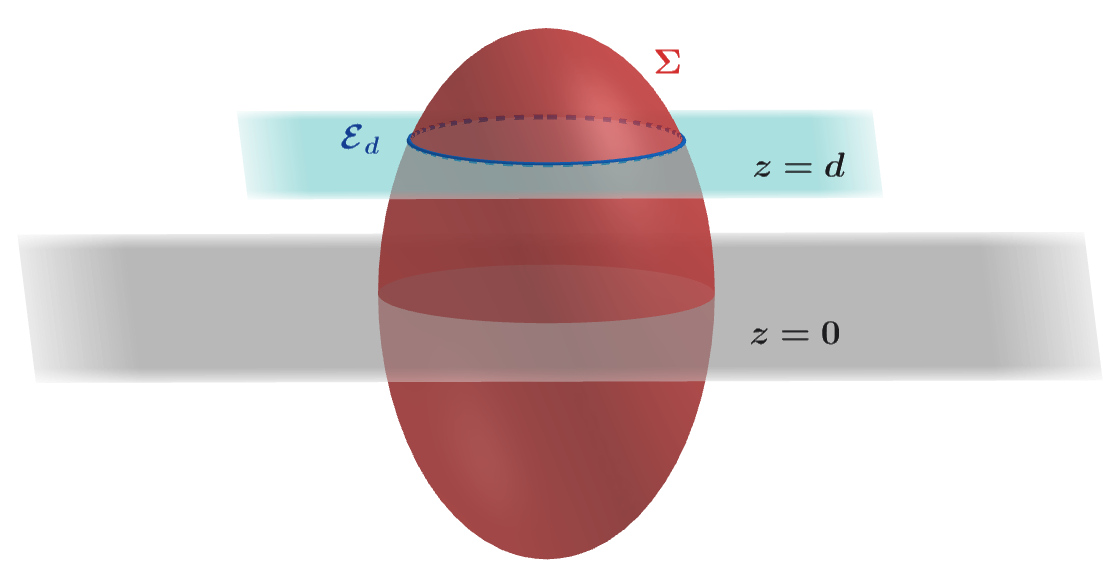}
	\label{FOTO3}\\
	\text{Source: Prepared by the authors.}
\end{figure}

\noindent This intersection yields an ellipse with the equation:
\begin{equation}\label{ellipse_intersection}
	\mathcal{E}_d: \quad \dfrac{x^2}{A^2} + \dfrac{y^2}{B^2} = 1,
\end{equation}
where the values of \( A \) and \( B \) are given by:
\begin{equation}\label{A e B}
	A = \left( \dfrac{a\sqrt{c^2 - d^2}}{c} \right) \quad \text{and} \quad B = \left( \dfrac{b\sqrt{c^2 - d^2}}{c} \right).
\end{equation}
Let \( \mathcal{E}_0 = \varphi^{-1}(\mathcal{E}_d) \) be the ellipse obtained by projecting the ellipse \( \mathcal{E}_d \) onto the plane \( z = 0 \) using the mapping given in \eqref{inversa_elipsoide}. Then, we obtain the following result.

\begin{figure}[H]
	\centering
	\caption{Projection of the ellipse $\mathcal{E}_d$ onto the $xy$-plane.}
	\includegraphics[width=13cm]{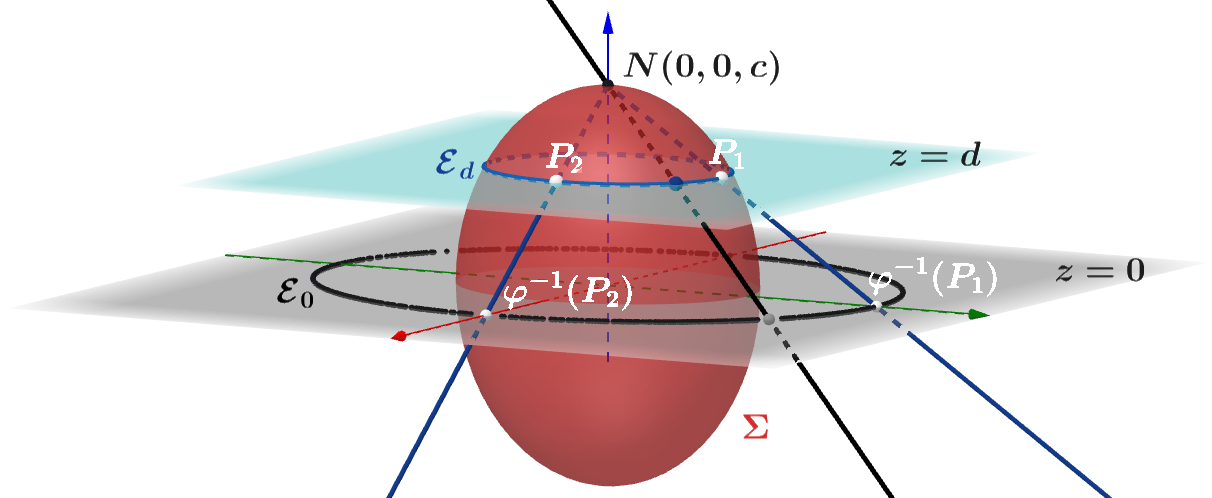}
	\label{FOTO4}\\
	\text{Source: Prepared by the authors.}
\end{figure}

\begin{theorem}\label{teorema1}
	The projection map \eqref{inversa_elipsoide} preserves eccentricity, that is, 
	$$
	e(\mathcal{E}_0) = e ( \mathcal{E}_d),
	$$
	where \( e(\mathcal{E}_d) \) and \( e(\mathcal{E}_0) \) denote the eccentricities of the ellipses \( \mathcal{E}_d \) and \( \mathcal{E}_0 \), respectively.
	
\end{theorem}

\begin{proof}
	
	We need to compute the eccentricity of the $\mathcal{E}_d $ and $ \mathcal{E}_0$. To do so, we assume without loss of generality that $A > B$, that is, $a > b$. Thus, the major axis is $2A$ and the minor axis is $2B$, and let $C$ represent half the distance between the focus. From the geometric definition of the ellipse, we have the relation:
	
	\begin{equation}\label{valor de C}
		C^2 = A^2 - B^2 = \dfrac{a^2(c^2 - d^2)}{c^2} - \dfrac{b^2(c^2 - d^2)}{c^2} = \dfrac{(a^2 - b^2)(c^2 - d^2)}{c^2},
	\end{equation}
	thus, the definition of eccentricity and \eqref{valor de C}, we obtain
	\begin{equation}\label{exec E_d}
		e(\mathcal{E}_d) = \dfrac{C}{A} = \dfrac{ \sqrt{\dfrac{(a^2 - b^2)(c^2 - d^2)}{c^2}} }{\dfrac{a\sqrt{(c^2-d^2)}}{c}} = \dfrac{\sqrt{a^2 - b^2}}{a}.
	\end{equation}
	Now, let us determine the eccentricity of the ellipse projected onto the plane by the inverse of the parameterization. For this, we need to obtain the equation of the projection of the ellipse generated by the intersection of the plane \( z = d \) with the ellipsoid.
	Let \( (x, 0, d) \in \Sigma \), then
	$$
	\begin{aligned}
		\dfrac{x^2}{a^2} + \dfrac{0^2}{b^2} + \dfrac{d^2}{c^2} = 1
		& \Longrightarrow  x = \dfrac{a}{c} \sqrt{c^2 - d^2}. 
	\end{aligned}
	$$
	Thus, \( (x, 0, d) = \left( \frac{a}{c} \sqrt{c^2 - d^2},\, 0,\, d \right) \in \Sigma \). Similarly, given \( (0, y, d) \in \Sigma \), we obtain \( y = \frac{b}{c} \sqrt{c^2 - d^2} \). Let us now analyze the projections of the points \( P_1\left( \frac{a}{c} \sqrt{c^2 - d^2},\, 0,\, d \right) \) and \( P_2\left( 0,\, \frac{b}{c} \sqrt{c^2 - d^2},\, d \right) \) under the inverse mapping \eqref{inversa_elipsoide}.In this way, we have
	\begin{align*}
		\varphi^{-1}(P_1)&= \varphi^{-1}\left( \dfrac{a}{c} \sqrt{c^2 - d^2},\, 0,\, d \right)\\
		&= \left( \dfrac{c \left( \dfrac{a}{c} \sqrt{c^2 - d^2} \right)}{c - d},\, \dfrac{c \cdot 0}{c - d} \right)\\
		&= \left( \dfrac{a \sqrt{c^2 - d^2}}{c - d},\, 0 \right)
	\end{align*}
	and
	\begin{align*}
		\varphi^{-1}(P_2)&= \varphi^{-1}\left( 0,\, \dfrac{b}{c} \sqrt{c^2 - d^2},\, d \right)\\
		&= \left( \dfrac{c \cdot 0}{c - d},\, \dfrac{c \left( \dfrac{b}{c} \sqrt{c^2 - d^2} \right)}{c - d} \right)\\
		&= \left( 0,\, \dfrac{b \sqrt{c^2 - d^2}}{c - d} \right).
	\end{align*}
	From these points, we obtain the ellipse in the \( xy \)-plane given by the equation
	\begin{equation}\label{elipsenoplanoxy}
		\mathcal{E}_0 :   \dfrac{x^2}{(A_0)^2} + \dfrac{y^2}{(B_0)^2} = 1.
	\end{equation}
	where, 
	\begin{equation}\label{A_0 e B_0}
		A_0 = \dfrac{a \sqrt{c^2 - d^2}}{c - d} \quad \text{and} \quad B_0 = \dfrac{b \sqrt{c^2 - d^2}}{c - d}.
	\end{equation}
	It is not difficult to see that the points projected by \( \varphi^{-1} \), obtained from the intersection of the ellipsoid with the plane \( z = d \), belong to \eqref{elipsenoplanoxy}. We now investigate the eccentricity of this new ellipse generated by the projection of the ellipse \eqref{ellipse_intersection}. Note that according to our initial assumption we have that $A_0 > B_0$ (because $a > b$) and consider \( C_0 \) to be half the distance between the focus. Then,
	
	\begin{equation}\label{valor de C'}
		(C_0)^2 = (A_0)^2 - (B_0)^2 = \dfrac{a^2(c^2 - d^2)}{(c - d)^2} - \dfrac{b^2(c^2 - d^2)}{(c - d)^2} = \dfrac{(a^2 - b^2)(c^2 - d^2)}{(c - d)^2}.
	\end{equation}
	Thus, the eccentricity of the $\mathcal{E}_0$ is given by 
	\begin{equation}\label{exce E_0}
		e(\mathcal{E}_0) = \dfrac{C_0}{A_0} = \dfrac{ \sqrt{\dfrac{(a^2 - b^2)(c^2 - d^2)}{(c - d)^2}} }{\dfrac{a \sqrt{(c^2 - d^2)}}{c - d}} = \dfrac{\sqrt{a^2 - b^2}}{a}.
	\end{equation}
	Therefore, from \eqref{exce E_0} and \eqref{exec E_d} follow the result.
	
\end{proof}

For the case of the elliptical paraboloid, we will consider the calculations completely analogous to the case of the ellipsoid with the elements defined in the image below. Considering the elliptical paraboloid given in \eqref{paraboloide} and the horizontal plane \( z = d \), with \( d < c \), the intersection is described by:
\begin{equation*}
	\mathcal{P}_d: 
	\begin{cases}
		z = c - \dfrac{x^2}{a^2} - \dfrac{y^2}{b^2} \\
		z = d
	\end{cases}, \quad d < c.
\end{equation*}

\begin{figure}[H]
	\centering
	\caption{Projection of the ellipse $\mathcal{P}_d$ onto the $xy$-plane.}
	\includegraphics[width=11cm]{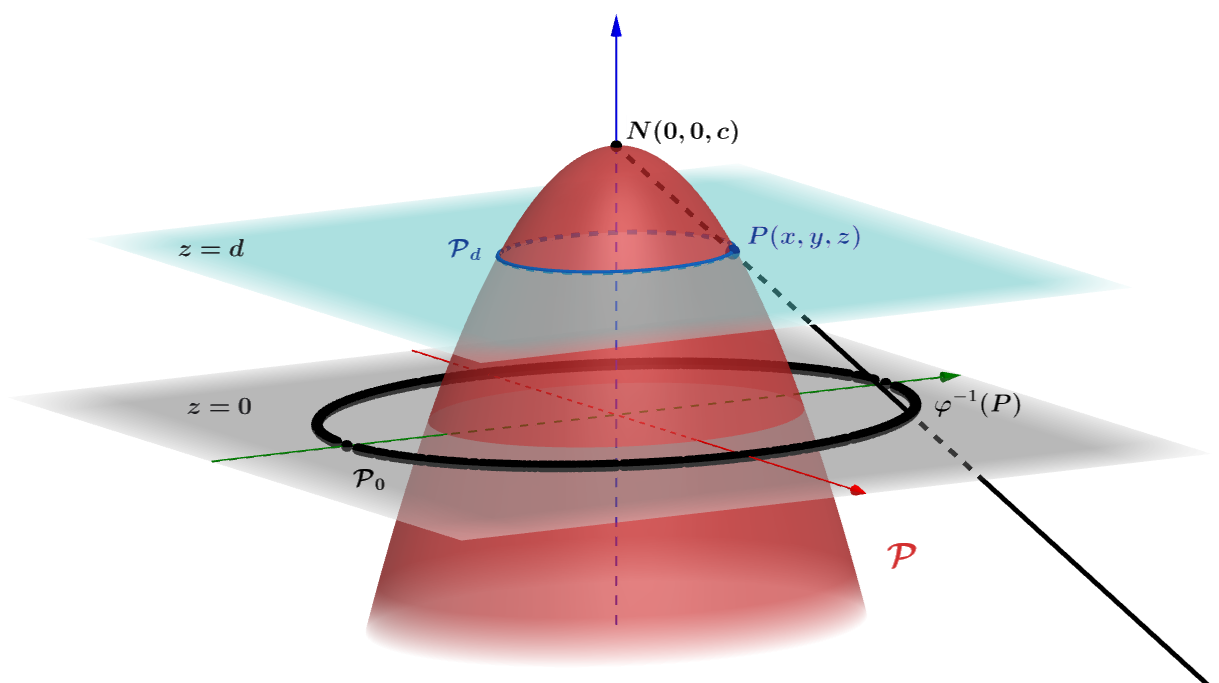}
	\label{FOTO5}\\
	\text{Source: Prepared by the authors.}
\end{figure}

This intersection yields an ellipse with the equation:
\begin{equation}\label{paraboloide_intersection}
	\mathcal{P}_d: \quad \dfrac{x^2}{A^2} + \dfrac{y^2}{B^2} = 1,
\end{equation}
where the values of \( A \) and \( B \) are given by:
\begin{equation}\label{A e B paraboloide}
	A = a\sqrt{c - d}  \quad \text{and} \quad B =  b\sqrt{c - d}.
\end{equation}
Let \( \mathcal{P}_0 = \phi^{-1}(\mathcal{P}_d) \) be the ellipse obtained by projecting the ellipse \( \mathcal{P}_d \) onto the plane \( z = 0 \) using the mapping given in \eqref{inversa_paraboloide}. Then, we obtain the following result.

\begin{theorem}\label{teorema2}
	The projection map \eqref{inversa_paraboloide} preserves eccentricity, that is, 
	$$
	e(\mathcal{P}_0) = e ( \mathcal{P}_d),
	$$
	where \( e(\mathcal{P}_d) \) and \( e(\mathcal{P}_0) \) denote the eccentricities of the ellipses \( \mathcal{P}_d \) and \( \mathcal{P}_0 \), respectively.
	
\end{theorem}

\begin{proof}
	
	We need to compute the eccentricity of the $\mathcal{P}_d $ and $ \mathcal{P}_0$. To do so, we assume without loss of generality that $A > B$, that is, $a > b$. Thus, the major axis is $2A$ and the minor axis is $2B$, and let $C$ represent half the distance between the focus. Doing a calculation similar to the previous case of the ellipsoid we obtain that
	$$
	C = \sqrt{(a^2 - b^2) (c-d)},
	$$
	and thus, the eccentricity of the ellipse $\mathcal{P}_d$ is given 
	\begin{equation}\label{excentr Pd}
		e(\mathcal{P}_d) = \dfrac{\sqrt{(a^2 - b^2) (c-d)}}{a\sqrt{c-d}} = \dfrac{\sqrt{a^2-b^2}}{a}.
	\end{equation}
	Now, let us determine the eccentricity of the ellipse projected onto the plane by \eqref{inversa_paraboloide}. For this, we need to obtain the equation of the projection of the ellipse $\mathcal{P}_d$ generated by the intersection of the plane \( z = d \) with the elliptical paraboloid. Proceeding analogously to the theorem \ref{teorema1}, we obtain the equation of the projected ellipse $\mathcal{P}_0$ given by
	\begin{equation}\label{elipsenoplanoxy_pd}
		\mathcal{P}_0 :   \dfrac{x^2}{(A_0)^2} + \dfrac{y^2}{(B_0)^2} = 1.
	\end{equation}
	where, 
	\begin{equation}\label{A_0 e B_0 paraboloide}
		A_0 = \dfrac{ca}{\sqrt{c - d}} \quad \text{and} \quad B_0 = \dfrac{cb }{\sqrt{c - d}}.
	\end{equation}
	We now investigate the eccentricity of this new ellipse generated by the projection of the ellipse \eqref{paraboloide_intersection}. Note that according to our initial assumption we have that $A_0 > B_0$ (because $a > b$) and consider \( C_0 \) to be half the distance between the focus. Then,
	\begin{equation}\label{valor de C0 pd}
		(C_0) = \sqrt{(A_0)^2 - (B_0)^2} = \dfrac{c\sqrt{a^2 - b^2}}{\sqrt{c - d}}.
	\end{equation}
	Thus, the eccentricity of the $\mathcal{P}_0$ is given by 
	\begin{equation}\label{exce P0}
		e(\mathcal{P}_0) = \dfrac{c\sqrt{a^2-b^2}}{\sqrt{c-d}} \cdot \dfrac{\sqrt{c-d}}{ca} = \dfrac{\sqrt{a^2 - b^2}}{a}.
	\end{equation}
	Therefore, from \eqref{excentr Pd} and \eqref{exce P0} follow the result.

\end{proof}


\vspace{0.5cm}

\begin{center}
	\centering
	\footnotesize
	\text{5.2. RELATIONSHIP BETWEEN THE CURVATURES}
\end{center}

We will now see a result where we were able to establish a relationship between the curvatures of the ellipses generated by the sections of quadraphic surfaces with horizontal planes and their respective projections on the $xy$-plane.

\begin{theorem}\label{teorema3}
	Let $k_{\mathcal{E}_d}$ and $k_{\mathcal{E}_0}$ be the curvatures of ellipses \eqref{ellipse_intersection} and \eqref{elipsenoplanoxy}, then
	$$
	k_{\mathcal{E}_0}(t) =  \left(\dfrac{c-d}{c} \right) \cdot k_{\mathcal{E}_d}(t).
	$$
\end{theorem}

\begin{proof}
	Let us begin by computing the curvature of the ellipse \( \mathcal{E}_d \). To do this, note that a parameterization for the planar curve  \eqref{ellipse_intersection} is given by
	\begin{equation}\label{parametriza elipse ed}
		\alpha(t) = (A \cos t, B \sin t), \quad t \in [0,2\pi],
	\end{equation}
	where
	$$
	A = \dfrac{a}{c}\sqrt{c^2 - d^2} \quad \text{and} \quad B = \dfrac{b}{c}\sqrt{c^2 - d^2}.
	$$
	Thus, we have that the first derivative gives us the tangent vector
	$$
	\alpha'(t) = (-A \sin t,\, B \cos t),
	$$
	and the second derivative gives us the acceleration vector
	$$
	\alpha''(t) = (-A \cos t, -B \sin t).
	$$
	From definition \ref{definicao curvatura}, the curvature $k_{\mathcal{E}_d}(t)$ of a planar curve $\mathcal{E}_d$ is given by
	\begin{align*}
		k_{\mathcal{E}_d}(t) & = \dfrac{AB \sin^2(t) + AB \cos^2(t)}{A^2 \sin^2(t) + B^2 \cos^2(t)}  = \dfrac{AB}{A^2 \sin^2(t) + B^2 \cos^2(t)}.
	\end{align*}
	Substituting the values of $A$ and $B$, we obtain
	\begin{align*}
		k_{\mathcal{E}_d}(t) &= \frac{\dfrac{ab(c^2 - d^2)}{c^2}}{\left( \dfrac{(c^2 - d^2)}{c^2} \right)^{3/2} (a^2 \sin^2 t + b^2 \cos^2 t)^{3/2}} \\
		&= \frac{ab}{\dfrac{(c^2 - d^2)^{1/2}}{c}} \cdot \frac{1}{(a^2 \sin^2 t + b^2 \cos^2 t)^{3/2}} \\
		&= \frac{abc}{\sqrt{c^2 - d^2}} \cdot \frac{1}{(a^2 \sin^2 t + b^2 \cos^2 t)^{3/2}}.
	\end{align*}
	We conclude that the curvature of \eqref{ellipse_intersection} at a given instant \( t \) is given by
	\begin{equation}\label{curvatura1}
		k_{\mathcal{E}_d}(t) = \frac{abc}{\sqrt{c^2 - d^2}} \cdot \frac{1}{(a^2 \sin^2 t + b^2 \cos^2 t)^{3/2}}.
	\end{equation}
	Analogously to what we did for the ellipse $\mathcal{E}_d$, we can parametrize the ellipse $\mathcal{E}_0$ given in \eqref{elipsenoplanoxy} by
	\begin{equation}\label{parametriza elipse e0}
		\beta(t) = (A_0 \cos t, B_0 \sin t), \quad t \in [0,2\pi],
	\end{equation}
	where,
	$$
	A_0 = \dfrac{a \sqrt{c^2 - d^2}}{c - d} \quad \text{and} \quad B_0 = \dfrac{b \sqrt{c^2 - d^2}}{c - d}.
	$$
	Defining $k_{\mathcal{E}_0}(t)$ as the curvature of this ellipse $\mathcal{E}_0$ at time $t$, we have
	$$
	k_{\mathcal{E}_0}(t) = \frac{A_0 B_0}{\left( (A_0)^2 \sin^2 t + (B_0)^2 \cos^2 t \right)^{3/2}},
	$$
	substituting the values of $A_0$ and $B_0$, we obtain
	\begin{align*}
		k_{\mathcal{E}_0}(t) &= \frac{\dfrac{ab(c^2 - d^2)}{(c - d)^2}}{\left( \dfrac{a^2(c^2 - d^2)}{(c - d)^2}\sin^2 t + \dfrac{b^2(c^2 - d^2)}{(c - d)^2}\cos^2 t \right)^{3/2}}\\
		&=  \frac{\dfrac{ab(c^2 - d^2)}{(c - d)^2}}{\left( \dfrac{c^2 - d^2}{(c - d)^2} (a^2 \sin^2 t + b^2 \cos^2 t) \right)^{3/2}},
	\end{align*}
	thus, we have
	\begin{align*}
		k_{\mathcal{E}_0}(t) &= \frac{ab(c^2 - d^2)}{(c - d)^2} \cdot \frac{1}{\left( \dfrac{c^2 - d^2}{(c - d)^2} \right)^{3/2}} \cdot \frac{1}{(a^2 \sin^2 t + b^2 \cos^2 t)^{3/2}}\\
		&= \frac{ab(c^2 - d^2)}{(c - d)^2} \cdot \frac{(c - d)^3}{(c^2 - d^2)^{3/2}} \cdot \frac{1}{(a^2 \sin^2 t + b^2 \cos^2 t)^{3/2}}.
	\end{align*}
	Therefore, from \eqref{curvatura1} we get 
	\begin{align*}
		k_{\mathcal{E}_0}(t) &= \frac{ab(c - d)}{\sqrt{c^2 - d^2}} \cdot \frac{1}{(a^2 \sin^2 t + b^2 \cos^2 t)^{3/2}}\\
		& = k_{\mathcal{E}_d}(t) - \left(  \dfrac{abd}{\sqrt{c^2 - d^2} \cdot (a^2 \sin^2 t + b^2 \cos^2 t)^{3/2} } \right).
	\end{align*}
	By simplifying a bit further, we can multiply the expression in parentheses by $\dfrac{c}{c} = 1$, and thus we obtain
	$$
	k_{\mathcal{E}_0}(t) = k_{\mathcal{E}_d}(t) - k_{\mathcal{E}_d}(t)\frac{d}{c} = \left(\frac{c-d}{c}\right) \cdot k_{\mathcal{E}_d}(t).
	$$
	
\end{proof}

With calculations very similar to the previous case, we were able to establish the same result for the elliptical paraboloid.

\begin{theorem}\label{teorema4}
	Let $k_{\mathcal{P}_d}$ and $k_{\mathcal{P}_0}$ be the curvatures of ellipses \eqref{ellipse_intersection} and \eqref{elipsenoplanoxy}, then
	\begin{equation}\label{conclusao curvatura paraboloide}
		k_{\mathcal{P}_0}(t) =  \left(\dfrac{c-d}{c} \right) \cdot k_{\mathcal{P}_d}(t).
	\end{equation}
\end{theorem}

\begin{proof}
	
	The proof follows in a manner analogous to the theorem \ref{teorema3}, in this case, we will be more brief in the details. Let us begin by computing the curvature of the ellipse \( \mathcal{P}_d \). To do this, note that a parameterization for the planar curve  \eqref{paraboloide_intersection} is given by
	\begin{equation}\label{parametriza paraboloide Pd}
		\alpha(t) = (A \cos t, B \sin t), \quad t \in [0,2\pi],
	\end{equation}
	where $A = a\sqrt{c - d}$ and $B = b\sqrt{c - d}.$ Thus, we have that the first and second derivative gives us the tangent and acceleration vectors
	$$
	\alpha'(t) = (-A \sin t,\, B \cos t), \quad \text{and} \quad \alpha''(t) = (-A \cos t, -B \sin t).
	$$
	From definition \ref{definicao curvatura} and following in an analogous manner to the theorem \eqref{teorema3} the curvature $k_{\mathcal{P}_d}(t)$ of a planar curve $\mathcal{P}_d$ is given by
	\begin{equation}
		k_{\mathcal{P}_d}(t) = \dfrac{ab}{\sqrt{c-d} ( a^2 \sin^2 t + b^2 \cos^2 t )^{3/2} }.
	\end{equation}
	Analogously to what we did for the ellipse $\mathcal{P}_d$, we can parametrize the ellipse $\mathcal{P}_0$ given in \eqref{elipsenoplanoxy_pd} by
	\begin{equation}\label{parametriza elipse p0}
		\beta(t) = (A_0 \cos t, B_0 \sin t), \quad t \in [0,2\pi],
	\end{equation}
	where,
	$$
	A_0 = \dfrac{ca}{\sqrt{c - d}} \quad \text{and} \quad B_0 = \dfrac{cb}{\sqrt{c - d}}.
	$$
	Defining $k_{\mathcal{P}_0}(t)$ as the curvature of this ellipse $\mathcal{P}_0$ at time $t$, from \eqref{definicao curvatura} and in a similar way to the Theorem \ref{teorema3} we get
	\begin{equation}\label{k0 paraboloide}
		k_{\mathcal{P}_0}(t) = \dfrac{ab \sqrt{c-d}}{c ( a^2 \sin^2 t + b^2 \cos^2 t )^{3/2}},
	\end{equation}
	multiplying and dividing \eqref{k0 paraboloide} by $\sqrt{c-d}$, we conclude the expected result \eqref{conclusao curvatura paraboloide}.
	
\end{proof}

\vspace{0.5cm} 

\begin{center}
	\centering
	\footnotesize
	\text{5.3. RELATIONSHIP BETWEEN THE ARC LENGTHS}
\end{center}

Next, we present the results that characterize the relationship between the arc lengths of the ellipses $\mathcal{E}_d$ and $\mathcal{E}_0$. Within the framework of differential geometry, we derive an expression that relates the length of the projected ellipse to the length of an ellipse arising from an arbitrary horizontal section of the ellipsoid, as well as of the elliptic paraboloid.

\begin{theorem}\label{teorema5}
	Let \( \mathcal{L}_{\mathcal{E}_0}(t) \) and \( \mathcal{L}_{\mathcal{E}_d}(t) \) be the arc length function of the ellipses \( \mathcal{E}_0 \) and \( \mathcal{E}_d \) respectively, then
	\begin{equation}
		\mathcal{L}_{\mathcal{E}_0}(t) = \left(\frac{c}{c-d} \right) \cdot \mathcal{L}_{\mathcal{E}_d}(t)
	\end{equation}
\end{theorem}

\begin{proof}
	We begin by computing the arc length of the curve defined by the ellipse $\mathcal{E}_d$. That is, we consider the parametrization of the curve given by \eqref{parametriza elipse ed}, from the definition of arc length, we obtain
	\begin{align}\label{L_Ed}
		\mathcal{L}_{\mathcal{E}_d}(t) & = \int_{0}^{2\pi} |\alpha'(t)|dt = \int_0^{2\pi} \sqrt{ A^2 \sin^2(t) + B^2 \cos^2(t) } dt \nonumber \\
		& =  \dfrac{\sqrt{c^2 - d^2}}{c} \int_0^{2\pi} \sqrt{ a^2 \sin^2(t) + b^2 \cos^2(t) },
	\end{align}
	on the other hand, we compute the arc length of the curve $\mathcal{E}_0$. To this end, we use the parametrization given by \eqref{parametriza elipse e0}. Thus,
	\begin{align}\label{L_E0}
		\mathcal{L}_{\mathcal{E}_0}(t) & = \int_{0}^{2\pi} |\beta'(t)|dt = \int_0^{2\pi} \sqrt{ (A_0)^2 \sin^2(t) + (B_0)^2 \cos^2(t) } dt \nonumber \\
		& =  \dfrac{\sqrt{c^2 - d^2}}{c-d} \int_0^{2\pi} \sqrt{ a^2 \sin^2(t) + b^2 \cos^2(t) },
	\end{align}
	by multiplying and dividing by a constant $c>0$ and comparing with \eqref{L_Ed}, we conclude that
	$$
	\mathcal{L}_{\mathcal{E}_0}(t) = \left( \frac{c}{c-d} \right) \cdot \mathcal{L}_{\mathcal{E}_d}(t).
	$$
	
\end{proof}

With calculations very similar to the previous case, we were able to establish the same result for the elliptical paraboloid.

\begin{theorem}\label{teorema6}
	Let \( \mathcal{L}_{\mathcal{P}_0}(t) \) and \( \mathcal{L}_{\mathcal{P}_d}(t) \) be the arc length function of the ellipses \( \mathcal{P}_0 \) and \( \mathcal{P}_d \) respectively, then
	\begin{equation}
		\mathcal{L}_{\mathcal{P}_0}(t) = \left(\frac{c}{c-d} \right) \cdot \mathcal{L}_{\mathcal{P}_d}(t)
	\end{equation}
\end{theorem}

\begin{proof}
	
	With accounts analogous to the theorem \ref{teorema5}, we were also able to establish the expected result.
	
\end{proof}


\vspace{0.5cm}

\begin{center}
	\centering
	\footnotesize
	\text{5.4. RELATIONSHIP BETWEEN THE AREAS}
\end{center}

Next, we present the results that describe how the areas of the ellipses $\mathcal{E}_d$ and $\mathcal{E}_0$ are related. In this context, we find a relation to calculate the area of the projected ellipse as a function of the area of an ellipse formed by an arbitrary horizontal section of the ellipsoid, as well as in the elliptic paraboloid.

\begin{theorem}\label{teorema7}
	Let \( \mathcal{A}(\mathcal{E}_0) \) and \( \mathcal{A}(\mathcal{E}_d) \) be the areas of the ellipses \( \mathcal{E}_0 \) and \( \mathcal{E}_d \) respectively, then
	\begin{equation}
		\mathcal{A}(\mathcal{E}_0) = \left(\frac{c}{c-d} \right)^{2} \cdot \mathcal{A}(\mathcal{E}_d)
	\end{equation}
\end{theorem}

\begin{proof}
	To prove this, initially recall that the areas of the ellipses are given by \( \mathcal{A}(\mathcal{E}_0) = A_0 B_0 \pi \) and \( \mathcal{A}(\mathcal{E}_d) = A B \pi \). Thus, from the values of \( A, B, A_0 \) and \( B_0 \) given in \eqref{A e B} e \eqref{A_0 e B_0}, we obtain that the area of the ellipse \( \mathcal{E}_d \) is given by
	\begin{equation}\label{area ad}
		\mathcal{A}(\mathcal{E}_d) = \frac{ab}{c^2} (c^2 - d^2) \pi,
	\end{equation}
	while the area of the ellipse \( \mathcal{E}_0 \) can be expressed as
	\begin{equation}\label{area a0}
		\mathcal{A}(\mathcal{E}_0) = \frac{ab}{(c - d)^2} (c^2 - d^2) \pi = \dfrac{c^2}{(c-d)^2} \mathcal{A}(\mathcal{E}_d),
	\end{equation}
	from \eqref{area ad}, note that \( c^2 A(\mathcal{E}_d) = ab (c^2 - d^2) \pi \). Thus, from \eqref{area a0} we obtain
	$$
	\mathcal{A}(\mathcal{E}_0) = \left( \dfrac{c}{c-d} \right)^2  \cdot \mathcal{A}(\mathcal{E}_d)
	$$

\end{proof}

Analogously to the previous cases, we also have the result for the elliptical paraboloid.

\begin{theorem}\label{teorema8}
	Let \( \mathcal{A}(\mathcal{P}_0) \) and \( \mathcal{A}(\mathcal{P}_d) \) be the areas of the ellipses \( \mathcal{P}_0 \) and \( \mathcal{P}_d \) respectively, then
	\begin{equation}
		\mathcal{A}(\mathcal{P}_0) = \left(\frac{c}{c-d} \right)^{2} \cdot \mathcal{A}(\mathcal{P}_d)
	\end{equation}
\end{theorem}

\begin{proof}
	
	With accounts analogous to the theorem \ref{teorema7}, we were also able to establish the expected result.
	
\end{proof}

\begin{remark}
	Note that the growth rates of the theorems that relate curvatures (Theorems \ref{teorema3} and \ref{teorema4}) are inversely proportional to the theorems for calculating the length of curves and areas (Theorems \ref{teorema5}, \ref{teorema6}, \ref{teorema7} and \ref{teorema8}).
	It is easy to see that when:
	$$
	k_{\mathcal{E}_0}(t) \longrightarrow 0 \Longrightarrow
	\begin{cases}
		\mathcal{L}_{\mathcal{E}_0}(t) \longrightarrow \infty \\
		\mathcal{A}(\mathcal{E}_0) \longrightarrow \infty
	\end{cases},
	$$
	on the orther hand, if $\mathcal{L}_{\mathcal{E}_0}(t)$ or $\mathcal{A}(\mathcal{E}_0)$ goes to zero, then $k_{\mathcal{E}_0}(t) \longrightarrow \infty$. The same happens with the results obtained in the case of the elliptical paraboloid. Note that this same phenomenon occurs in the case of the sphere, where we know that its curvature is given by 1 over its radius, that is, the smaller the curvature, the greater its area.
\end{remark}
	

	\bibliographystyle{amsplain}
	
\end{document}